\documentclass[a4paper,12pt,leqno]{amsart}
\usepackage[colorlinks=true,linkcolor=darkblue,citecolor=darkblue]{hyperref}

\usepackage{graphicx}
\usepackage{pinlabel} 
\usepackage{amsmath}
\usepackage{amsfonts}
\usepackage{amssymb}
\usepackage{mathtools}
\usepackage[all,cmtip]{xy}
\usepackage{amsthm}

\usepackage{tikz-cd}

\usepackage{comment}
\usepackage{enumerate}
\usepackage{url}
\usepackage{epsfig}
\usepackage{hyperref}
\usepackage[utf8]{inputenc}
\usepackage[capitalise]{cleveref}
\usepackage{geometry}
\makeatletter
\providecommand\@dotsep{5}
\def\listtodoname{List of Todos}
\def\listoftodos{\@starttoc{tdo}\listtodoname}
\makeatother

\newtheorem{theorem}{Theorem}[section]
\newtheorem{proposition}[theorem]{Proposition}
\newtheorem{corollary}[theorem]{Corollary}
\newtheorem{lemma}[theorem]{Lemma}

\newcommand{\mycomment}[1]{}

  \theoremstyle{definition}
\newtheorem{definition}[theorem]{Definition}

\newtheorem{remark}{Remark}

\newcommand{\Z}{\mathbb{Z}}

\newcommand{\calF}{{\mathcal F}}

\newcommand{\calG}{{\mathcal G}}

\newcommand{\nbeq}{\begin{equation}}
\newcommand{\neeq}{\end{equation}}
\newcommand{\beq}{\begin{equation*}}
\newcommand{\eeq}{\end{equation*}}

\DeclareMathOperator{\cd}{cd}

\DeclareMathOperator{\vcd}{vcd}

\newcommand{\ame}{Ame}

\DeclareMathOperator{\gd}{gd}
\DeclareMathOperator{\gdfin}{gd_{\calF_0}}
\DeclareMathOperator{\gdvc}{\underline{\underline{gd}}}

\DeclareMathOperator{\mcg}{Mod}
\DeclareMathOperator{\pmcg}{PMod}

\DeclareMathOperator{\pb}{PB}

\DeclareMathOperator{\rank}{rank}

\definecolor{darkblue}{rgb}{0.0, 0.0, 0.55}
\usepackage{verbatim}

\begin{document}

\title[]{Classifying spaces for families of virtually abelian subgroups of surface braid groups}



\author[Ramón Flores]{Ramón Flores}
\address{Dpto. Geometría y Topología, Facultad de Matemáticas, Instituto de Matemáticas (IMUS), Universidad de Sevilla, Av. Reina Mercedes s/n, 41012 Sevilla, Spain}
\email{ramonjflores@us.es}

\author[Juan González-Meneses]{Juan González-Meneses}
\address{Dpto. Álgebra, Facultad de Matemáticas, Instituto de Matemáticas (IMUS), Universidad de Sevilla, Av. Reina Mercedes s/n, 41012 Sevilla, Spain}
\email{meneses@us.es}

\author[Porfirio L. León Álvarez]{Porfirio L. León Álvarez}
\address{Instituto de Matemáticas, Universidad Nacional Autónoma de México, Oaxaca de Juárez, Oaxaca, México 68000}
\email{porfirio.leon@im.unam.mx}


\date{}



\begin{abstract}

Given a group $G$ and an integer $n \geq 0$, let $\mathcal{F}_n$ denote the family of all virtually abelian subgroups of $G$ of rank at most $n$.
In this article, we show that for each $n \geq 1$, the minimal dimension of a model for the classifying space $E_{\mathcal{F}_n}G$ for the pure braid group of a surface of non-negative Euler characteristic with at least one boundary component or one puncture is equal to the virtual cohomological dimension of $G$ plus $n$. We prove analogous results for the full braid group of the sphere and the pure braid group of the projective plane.
As an application, we compute the minimal dimension of a model for the classifying space associated to the family of amenable subgroups of pure surface braid groups.
\end{abstract}
\maketitle


\section{Introduction}
\noindent Let \( G \) be a group. A collection \( \calF \) of subgroups of \( G \) is called a \emph{family} if it is non-empty and closed under conjugation and taking subgroups. For a given family $\calF$ of subgroups of $G$, a \( G \)-CW-complex \( X \) is  a \emph{model for the classifying space} \( E_{\calF}G \) if all its isotropy groups belong to \( \calF \), and for any \( G \)-CW-complex \( Y \) whose isotropy groups are also in \( \calF \), there exists a unique \( G \)-map \( Y \to X \), up to \( G \)-homotopy. It can be shown that a model for the classifying space \( E_{\calF}G \) always exists and is unique up to \( G \)-homotopy equivalence. The \emph{\( \calF \)-geometric dimension} of \( G \) is defined as
\[
\gd_{\calF}(G) = \min\{n \in \mathbb{N} \mid \text{there exists a model for } E_{\calF}G \text{ of dimension } n\}.
\]

The \( \calF \)-geometric dimension has an algebraic counterpart, the \emph{\( \calF \)-cohomological dimension} \( \cd_{\calF}(G) \), which can be defined using Bredon cohomology; see for example \cite{MV03}. The $\calF$-geometric and $\calF$-cohomological dimensions are related by the following inequality (see \cite[Theorem 0.1]{LM00}):
\[
\cd_{\calF}(G) \leq \gd_{\calF}(G) \leq \max\{\cd_{\calF}(G), 3\}.
\]
In particular, if \( \cd_{\calF}(G) \geq 3 \), then \( \cd_{\calF}(G) = \gd_{\calF}(G) \). However, this equality does not always hold. For instance, in the case of the family \( \calF_0 \) of finite subgroups, it was shown in \cite{NB} that there exists a right-angled Coxeter group \( W \) such that \( \cd_{\calF_0}(W) = 2 \) and \( \gd_{\calF_0}(W) = 3 \). For other examples see \cite{SS19}.

Let \( n \geq 0 \) be an integer. A group is said to be \emph{virtually \( \mathbb{Z}^n \)} if it contains a finite-index subgroup isomorphic to \( \mathbb{Z}^n \). We define the family
\[
\calF_n = \{H \leq G \mid H \text{ is virtually } \mathbb{Z}^r \text{ for some } 0 \leq r \leq n\}.
\]

The classifying spaces \( E_{\calF_0}G \) and \( E_{\calF_1}G \) are of particular interest because they appear explicitly in the formulation of the Baum-Connes and Farrell-Jones isomorphism conjectures, respectively. For more details, see \cite{LR05}.

The classifying spaces \( E_{\calF_n}G \) have been extensively studied in recent years. In 2018, Nucinkis and collaborators in \cite{victor:nucinkis:corob} proved that for a free abelian group \( G \) of rank \( n \), the inequality \( \gd_{\calF_k}(G) \leq n + k \) holds; later, in \cite{PL}, was proven that this upper bound is sharp. 
In 2021, Tomasz in \cite{tomasz} studied  CAT(0)-groups. He shows that for a group \( G \) acting on a CAT(0) space of topological dimension \( n \), and under the assumption that \( G \) satisfies condition C, the inequality
\[
\cd_{\calF_r}(G) \leq n + r + 1
\]
holds for all \( 0 \leq r \leq n \).

In 2023, the following was proven in \cite{Rita:Porfirio:Luis}. Denote the mapping class group by \( \mcg(S) \) of a connected, orientable, compact surface \( S \), possibly with a finite number of punctures. If \( S \) has negative Euler characteristic, then for all \( n \in \mathbb{N} \), the geometric dimension satisfies
\[
\gd_{\calF_n}(\mcg(S)) \leq \cd(\mcg(S)) + n.
\]
In this article, we show that for the pure braid group $G$ of a surface, the minimal dimension of a model for \(E_{\mathcal{F}_n}G\) is equal to \(\mathrm{cd}(G) + n\).  In what follows, we state this result precisely. 

\subsection*{Virtually abelian dimension of finite-type surface braid groups}

Let \( \Sigma \) be a connected surface, possibly of infinite type (i.e., with non-finitely generated fundamental group), possibly non-orientable, and possibly with punctures or boundary. 
Let \( n \in \mathbb{N} \). 
The \emph{configuration space of \( n \) points on \( \Sigma \)}, denoted \( F_n(\Sigma) \), is defined as
\[
F_n(\Sigma) = \{ (x_1, x_2, \ldots, x_n) \in \Sigma^n \mid x_i \neq x_j \text{ for } i \neq j \}.
\]
Let \( \mathfrak{S}_n \) denote the symmetric group on \( n \) letters. There is a natural free action of \( \mathfrak{S}_n \) on \( F_n(\Sigma) \) by permuting the coordinates. 

The \emph{pure braid group} on \( \Sigma \) with \( n \) strands, denoted \( \pb_n(\Sigma) \), is defined as the fundamental group of the configuration space \( F_n(\Sigma) \). The \emph{braid group} on \( \Sigma \) with \( n \) strands, denoted \( B_n(\Sigma) \), is the fundamental group of the quotient space \( F_n(\Sigma) / \mathfrak{S}_n \). The following short exact sequence relates these two groups:
\[
1 \to \pb_n(\Sigma) \to B_n(\Sigma) \to \mathfrak{S}_n \to 1.
\]
Let $\Sigma$ be a surface with boundary $\partial\Sigma$, then $B_n(\Sigma)$ is isomorphic to $B_n(\Sigma-\partial\Sigma)$, see for example \cite[Remarks 8]{MR3382024}.
It has been established that normally poly-free groups satisfy the Farrell-Jones conjecture \cite[Theorem A]{MR4246787}, see also \cite{MR1797585}, \cite[Theorem 1.1]{MR4565703}, and \cite[Theorem 2.3.7]{JPSS}. In particular, the pure braid group of a surface with non-empty boundary is known to be normally poly-free \cite{MR1797585}. Moreover, in the case of surfaces of infinite type, the pure braid group is also normally poly-free \cite[Theorem~1.5(b)]{2025arXiv250610706L}, and then these groups also satisfy the Farrell-Jones conjecture.

One of the main results of this paper is the following.


\begin{theorem}\label{gd:f1:surface:braid}
Let \( \Sigma \) be a connected surface of finite type, possibly non-orientable and possibly with punctures or boundary components, with non-positive Euler characteristic and at least one puncture or one boundary component. Then for all integers $r\ge0$ 

\[
\gd_{\calF_r}(\pb_n(\Sigma)) =\cd(\pb_n(\Sigma)) + r.
\]

If, on the other hand, \( \Sigma \) is closed and has non-positive Euler characteristic, then the following inequalities hold:

\[
\cd(\pb_n(\Sigma)) + r - 1 \leq \gd_{\calF_r}(\pb_n(\Sigma)) \leq\cd(\pb_n(\Sigma)) + r
\]
for all integers $r\ge0$.
\end{theorem}

For the full braid group of the sphere, we obtain an analogous result.

\begin{theorem}\label{abelian:dim:sphere:0}
Let $\Sigma$ be the sphere. For all integers $r\geq 0$
\[
  \gd_{\calF_r}(B_n(\Sigma)) = \vcd(B_n(\Sigma)) + r.
\]
\end{theorem}

\begin{remark}
When \(\chi(\Sigma)>0\), the \(\mathcal{F}_n\)-dimension is known only in certain cases. In particular, for the braid group on the disk it was computed in \cite{Ramon:Juan} and \cite[Theorem~1.3]{PL}. The computation of the \(\mathcal{F}_n\)-dimension in \cref{gd:f1:surface:braid} is explicit since the cohomological dimension of surface braid groups has already been established in the literature; see, for example, \cite[Theorem~1.1]{2025arXiv250610706L} and the references therein. 
\end{remark}

To prove \cref{gd:f1:surface:braid}, we obtain the upper bound by applying the Lück--Weiermann construction and the  Fadell--Neuwirth fibration to analyze the geometric dimension for proper actions of the Weyl groups associated to abelian subgroups.

For the lower bound, we show that the surface braid group admits an abelian subgroup of rank equal to its cohomological dimension. This fact may be of independent interest.

\begin{theorem}\label{free:abelian:subgroup:cd:0}
Let \( \Sigma \) be a connected surface of finite type, possibly non-orientable, with a non-positive Euler characteristic. Then the pure braid group $\pb_n(\Sigma)$ contains a free abelian subgroup of rank $n$.
\end{theorem}

\subsection*{Amenable dimension of surface braid groups.} 

As an application of \cref{gd:f1:surface:braid} we obtain the following. 
It is well known that surface braid groups satisfy the strong Tits alternative, 
that is, every subgroup is either virtually abelian or contains a non-abelian free group of rank~2 
(see \cite{MR745513,MR800253,MR726319}). Then every amenable subgroup is virtually free abelian. Consequently, the family of amenable subgroups coincides with the family of virtually abelian subgroups, 
and we obtain the following corollary.

\begin{corollary}
    Let \( \Sigma \) be a connected surface of finite type with at least one boundary component and negative Euler characteristic. Then 
    \[
    \gd_{\ame}\!\big(\pb_n(\Sigma)\big) =\cd\!\big(\pb_n(\Sigma)\big) + n.
    \]
\end{corollary}

The classifying space for the family of amenable subgroups has gained increasing attention in recent years due to its connection with the computation of the amenable category (see \cite{MR4474790}) and with the study of the amenable covering dimension (see \cite{MR4162286} \cite{2022arXiv220205606L}).

\subsection*{Virtually cyclic dimension of infinite-type surface braid groups}
When the surface is of infinite type, we can also compute the virtually cyclic geometric dimension.  

\begin{theorem}
Let \( \Sigma \) be a connected surface of infinite type. Then
\[
\gd_{\mathcal{F}_1}\bigl(\pb_n(\Sigma)\bigr) = n + 1.
\]
\end{theorem}

To establish the theorem above, it was necessary to show that the pure braid group \(\pb_n(\Sigma)\) satisfies condition~C. Namely, for every element of infinite order \(h \in \pb_n(\Sigma)\) and every \(g \in \pb_n(\Sigma)\), the equation 
\[
g h^l g^{-1} = h^m
\]
implies \(\lvert l \rvert = \lvert m \rvert\). Once this condition is verified, \cref{luck:condition:C} reduces the problem to computing the proper geometric dimension of the Weyl groups of virtually cyclic subgroups.

\subsection*{Outline of the paper}
In \cref{prelimi} we collect classical results on surface braid groups that will be used later, most notably the Fadell--Neuwirth fibration. In the same section we also recall the L\"uck--Weiermann pushout construction, which we use to build classifying spaces \(E_{\mathcal{F}_n}G\) inductively. In \cref{virt:abelian:dimension:surface:braid:groups} we prove our main result, \cref{gd:f1:surface:braid}, computing the \(\mathcal{F}_n\)-geometric dimension. Finally, in \cref{virt:abelian:dimension:surface:braid:groups:inf} we determine the virtually cyclic dimension for infinite-type surface braid groups.

\subsection*{Acknowledgments} The third author thanks Luis Jorge Sánchez Saldaña 
and Jesús Hernández Hernández for helpful discussions, and the Institute for Research and the Department of Geometry and Topology of the University of Seville, when part of this work was completed. The first author was partially supported by the project PID2024-155800NC32, and the second author was partially supported by the project PID2024-157173NBI00, both funded by MCIN/AEI/10.13039/501100011033 and by FEDER, EU. The second author also aknowledges partial support from “Programa operativo PPIT-FEDER Andalucía 2021-2027" SOL2024-31596 and SOL2024-31708. The third author's work 
was supported by the UNAM Postdoctoral Program (POSDOC) and by DGAPA-UNAM through 
grant PAPIIT~IN102426.

\section{Preliminaries}\label{prelimi}

\subsection{Surface braid groups} In this subsection we recall some results on surface braid groups that we need later.

\begin{proposition}\cite{Fadell1962}\cite[Lemma 1.27, Exercise 1.4.1 ]{kassel2008}
Let \( \Sigma \) be a connected surface, possibly non-orientable or of infinite type, possibly with punctures, and without boundary. The map
\[
\rho: F_{m+1}(\Sigma) \to F_m(\Sigma),
\]
defined by forgetting the last point, i.e.,
\[
\rho(x_1, x_2, \dots, x_{m+1}) = (x_1, x_2, \dots, x_m),
\]
is a fibration with fiber homeomorphic to \( \Sigma - Q_m \) where \( Q_m = \{p_1, p_2, \dots, p_m\} \) is a set of \( m \) distinct points on \( \Sigma \).
\end{proposition}

This fibration can be used to prove that the configuration spaces \( F_m(\Sigma) \) are Eilenberg-MacLane spaces. Using this fact and the associated long exact sequence in homotopy groups, we obtain the following result:

\begin{theorem}[Fadell-Neuwirth]
\label{Fadell:Neuwirth}
Let \(\Sigma\) be a connected surface without boundary, which may be non-orientable, of infinite type, or with punctures, and assume further that \(\Sigma\) is neither the two-sphere nor the real projective plane. Then, there exists the following short exact sequence:
\[
1 \to \pi_1(\Sigma - Q_n) \to \pb_{n+1}(\Sigma) \xrightarrow{\varphi} \pb_n(\Sigma) \to 1,
\]
where \( Q_n \) denotes a set of \( n \) distinct points on \( \Sigma \).
\end{theorem}
The virtual cohomological dimension of surface braid groups has been established in the literature; see, for example, \cite[Theorem 1.1]{2025arXiv250610706L} and the references therein.
\begin{theorem}\label{vcd:surfaces:braid:groups:0}
Let $\Sigma_{g,k}^r$ denote a connected surface of finite type of genus $g$, with $k$ punctures and $r$ boundary components, and let $n \geq 1$. The virtual cohomological dimension of \( \mathrm{B}_n(\Sigma_{g,k}^r) \) is given by:
\[
\vcd(\mathrm{B}_n(\Sigma_{g,k}^r)) = 
\begin{cases}
    n - 2 & \text{if } g = -1, \, k + r = 0, \, n \geq 3, \\
    n - 3 & \text{if } g = 0, \, k + r = 0, \, n \geq 4, \\
    n - 1 & \text{if } g = 0, \, k + r = 1, \, n \geq 1, \\
    n & \text{if } g = 0, \, k + r \geq 2, \, n \geq 1, \\
    n + 1 & \text{if } g = 1, \, k + r = 0, \, n \geq 1, \\
    n + 1 & \text{if } \lvert g \rvert \geq 2, \, k + r = 0, \, n \geq 1, \\
    n & \text{if } \lvert g \rvert \geq 1, \, k + r \geq 1, \, n \geq 1, \\
    0 & \text{otherwise}.
\end{cases}
\]
\end{theorem}

\begin{theorem}\cite[Corollary 4.4]{2025arXiv250610706L}\label{dc:infi:braid}
Let $\Sigma$ be a  connected surface of infinite type, then $\cd(\mathrm{B}_n(\Sigma))=n$.
\end{theorem}

\subsection{Lück-Weiermann  construction}
In this subsection we introduce a particular construction of Lück-Weiermann \cite[Theorem 2.3]{LW12} that we need later.

 \begin{definition}
     
 Let $\calF\subset\calG$ be two families of subgroups of $G$. Let $\sim$ be  an equivalence relation in $\calG-\calF$. We say that $\sim$ is strong  if the following is satisfied 
 
\begin{enumerate}[a)]
    \item If  $H, K \in \calG-\calF$ with $H\subseteq K$, then $H\sim K$;
    \item If $H, K \in \calG-\calF$ and $g\in G$, then $H\sim K$ if and only if $gHg^{-1} \sim gKg^{-1}$.
\end{enumerate}
 \end{definition}
 
 \begin{definition}
     Let $G$ be a group and $L, K$ be subgroups of $G$. We say that $L$ and $K$ are commensurable if $L\cap K$ has finite index in both $L$ and $K$.
 \end{definition}
 
\begin{definition}
Let $G$ be a group and let $H$ be a subgroup of $G$. We define the commensurator of $H$ in $G$ as  $$N_{G}[H]:=\{g\in G \mid gHg^{-1}\text{ is commensurable with }H \}.$$
\end{definition}

      Let $G$ be a group, $H$ a subgroup of $G$ and $n\geq0$. Consider the following nested families of $G$,  ${\calF_n}\subseteq {\calF_{n+1}}$, let $\sim$ the equivalence relation in ${\calF_{n+1}}-{\calF_n}$ given by commensurability. It is easy to check that this is a strong equivalence relation.

We introduce the following notation:
\begin{itemize}
    \item We denote by $({\calF_{n+1}}-{\calF_n})/\sim$ the equivalence classes  in  ${\calF_{n+1}}-{\calF_n}$. Given $L\in ({\calF_{n+1}}-{\calF_n})$ we denote by $[L]$ its equivalence class.
    
    \item Given $[L]\in ({\calF_{n+1}}-{\calF_n})/\sim$, we define the next family of subgroups of  $N_{G}[L]$ $${\calF_{n+1}}[L]:=\{K\leq N_G[L]| K\in  ({\calF_{n+1}}-{\calF_n}), [K]=[L]\}\cup ({\calF_n}\cap N_G[L]).$$
    
\end{itemize}

\begin{theorem}\cite[Theorem 2.3]{LW12}\label{Luck:weiermann}  Let $G$ be a group. Let $I$ be a complete set of representatives of conjugation classes in $(\calF_{n+1}-\calF_{n})/\sim$. Choose arbitrary $N_G[H]$-CW-models for $E_{\calF_{n}\cap N_{G}[H]}N_{G}[H]$ and $E_{ \calF_{n+1}[H]}N_{G}[H]$, and an arbitrary model for  $E_{\mathcal{F}_{n}}G$. Consider the  following $G$-push-out 
  \[
\begin{tikzpicture}
  \matrix (m) [matrix of math nodes,row sep=3em,column sep=4em,minimum width=2em]
  {
     \displaystyle\bigsqcup_{[H]\in I} G\times_{N_{G}[H]}E_{\calF_{n}\cap N_{G}[H]}N_{G}[H] & E_{\calF_n}G \\
      \displaystyle\bigsqcup_{[H]\in I} G\times_{N_{G}[H]}E_{ \calF_{n+1}[H]}N_{G}[H] & X \\};
  \path[-stealth]
    (m-1-1) edge node [left] {$\displaystyle\bigsqcup_{[H]\in I}id_{G}\times_{N_G}f_{[H]}$} (m-2-1) (m-1-1.east|-m-1-2) edge  node [above] {$i$} (m-1-2)
    (m-2-1.east|-m-2-2) edge node [below] {} (m-2-2)
    (m-1-2) edge node [right] {} (m-2-2);
\end{tikzpicture}
\]
such that $f_{[H]}$ is cellular $N_{G}[H]$-map for every $[H]\in I$ and either (1) $i$ is an inclusion of $G$-CW-complexes, or (2) such that every map $f_{[H]}$ is an inclusion of $N_{G}[H]$-CW-complexes for every $[H]\in I$ and $i$ is a cellular $G$-map. Then $X$ is a model for $E_{\calF_{n+1}}G$.
 \end{theorem}
 \begin{remark}\label{dimension:LW}
The conditions in \cref{Luck:weiermann} are not restrictive. For instance, to satisfy the  condition $(2)$, we can use the equivariant cellular approximation theorem to assume that the maps $i$ and $f_{[L]}$  are cellular maps for all $[L]\in I$,  and to make the function $f_{[L]}$ an inclusion for every $[L]\in I$,  we can replace the spaces by the mapping cylinders. See \cite[Remark 2.5]{LW12}. 
\end{remark}
Following the notation of \cref{Luck:weiermann} we have \begin{corollary}\label{lw:gd:upper:bound}
$\gd_{\calF_{n+1}}(G)\leq \max\{ \gd_{\calF_{n}}(G)+1, \gd_{ \calF_{n+1}[L]}(N_{G}[L])\mid L\in I\}.$
\end{corollary}
\subsection*{The push-out of a union of families}
 The following lemma will be also useful.
\begin{lemma}\cite[Lemma~4.4]{DQR11}\label{lemma:union:families}
Let $G$ be a group and $\calF,\, \calG$ be two families of subgroups of $G$.  Choose arbitrary $G$-$CW$-models for $E_{\calF} G$, $E_{\calG}G$ and $E_{\calF\cap\calG} G$. Then, the $G$-$CW$-complex $X$ given by the cellular homotopy $G$-push-out
\[
\xymatrix{
E_{\calF\cap\calG}G \ar[r] \ar[d] & E_{\calF}G \ar[d]\\
E_{\calG}G \ar[r] & X
}
\]
is a model for $E_{\calF\cup\calG}G$.
\end{lemma}
With the notation \cref{lemma:union:families} we have the following
\begin{corollary}\label{upper:bound:gd:two:families}
  $\gd_{\calG\cup\calF}(G)\leq\max\{\gd_{\calF}(G), \gd_{\calG}(G), \gd_{\calG\cap\calF}(G)+1\}.$  
\end{corollary}
\subsection*{Nested families}

Given a group $G$ and two nested families $\calF\subseteq \calG$ of $G$, we will use the following propositions to bound the geometric dimension $\gd_{\calF}(G)$ using the geometric dimension $\gd_{\calG}(G)$.

\begin{proposition}\cite[Proposition 5.1 (i)]{LW12}\label{nested:families:lw}
 Let $G$ be a group and let $\calF$ and $\calG$ be
two families of subgroups such that $\calF\subseteq \calG$. Suppose for every $H\in \calG$ we have $\gd_{\calF\cap H}(H)\leq d$. Then $\gd_{\calF}(G)\leq \gd_{\calG}(G)+d.$
\end{proposition}
\section{Classifying spaces of finite-type surface braid groups}\label{virt:abelian:dimension:surface:braid:groups}

In this section we compute the virtually abelian dimension of finite-type surface braid groups.

Before computing this dimension we need to prove some lemmas.
Let $G$ be a group and $H<G$. We denote $C_G(H)$, $N_G(H)$, and $W_G(H)=N_G(H)/H$  the centralizer, normalizer, and the Weyl group of $H$ respectively. 

Let $\Sigma_{g,k}^r$ denote a connected surface of finite type of genus $g$, with $k$ punctures and $r$ boundary components. In \cite[Remarks 8]{MR3382024}, it was observed that the inclusion $i\colon \Sigma_{g,k}^r \to \Sigma_{g,k+r}$ is a homotopy equivalence, and this induces a homotopy equivalence between their $n$-configuration spaces. In particular: 
\begin{lemma}\label{surfacebraid:puntures:boundaries}
    Let $\Sigma_{g,k}^r$ be the connected surface of genus $g$ with $r$ boundary components and  $k$ punctures. Then $\pb_n(\Sigma_{g,k}^r) \cong \pb_n(\Sigma_{g,k+r})$.
\end{lemma}

\begin{lemma}\label{property:pass:subgroup}
Let $G$ be a group such that for every virtually free abelian subgroup $H \leq G$
there exists a subgroup $H' \leq G$ commensurable with $H$ satisfying
$N_G[H] = N_G(H')$. Then every subgroup of $G$ has the same property.
\end{lemma}

\begin{proof}
Let $L \leq G$ and let $H \leq L$ be a virtually free abelian subgroup. By
hypothesis, there exists a subgroup $H' \leq G$ commensurable with $H$ such that
$N_G[H] = N_G(H')$. It follows that
\[
  N_L[H] = L \cap N_G[H] = L \cap N_G(H') \subseteq N_L(H' \cap L).
\]
We claim that $H' \cap L$ is commensurable with $H$. Indeed, $H \cap (H' \cap L)
= H \cap H'$, which is commensurable with $H$ by hypothesis. It remains to show
that $H \cap H'$ has finite index in $H' \cap L$. This follows from the inequality
\[
  [H' \cap L : H \cap H'] \leq [H' : H \cap H']  < \infty,
\]
which holds since $H$ and $H'$ are commensurable. This proves the claim.

Since $H' \cap L$ is commensurable with $H$ and $H \leq L$, we have $N_L[H] =
N_L[H' \cap L]$. Combined with the inclusion established above,
\[
  N_L[H' \cap L] = N_L[H] \subseteq N_L(H' \cap L).
\]
Since the reverse inclusion $N_L(H' \cap L) \subseteq N_L[H' \cap L]$ always
holds, we conclude that $N_L(H' \cap L) = N_L[H' \cap L]$, as required.
\end{proof}

The following result will be used throughout the remainder of the paper.

\begin{theorem}\cite[Proposition 2.2]{Paris:Rolfsen}\label{lemma:subsurface-injective}
Let \( M \) be a surface different from the sphere and the projective plane, and let \( N \subseteq M \) be a subsurface defined as the closure of an open subset of \( M \), such that every boundary component of \( N \) is either a boundary component of \( M \) or lies in the interior of \( M \). Assume also that none of the connected components of the closure of \( M - N \) is a disk. Then the induced morphism
\[
\psi \colon B_n(N) \to B_n(M)
\]
is injective.
\end{theorem}

\begin{remark}
The conditions imposed on the subsurface \( N \subseteq M \) follow the convention established in~\cite{Paris:Rolfsen}, ensuring that every boundary component of \( N \) is either a boundary component of \( M \) or fully contained in the interior of \( M \).
\end{remark}

\begin{lemma}\label{commen:equal:norma}Let $\Sigma$ be a surface of finite type (possibly with finitely many punctures) and \(\chi(\Sigma) \leq 0\).
    For every infinite virtually abelian subgroup \(L\) of \(\pb_{m}(\Sigma)\), there is $H<\pb_{m}(\Sigma)$ commensurable with $L$ such that $N_{\pb_{m}(\Sigma)}[L]=N_{\pb_{m}(\Sigma)}(H).$
\end{lemma}
\begin{proof}
First suppose that \(\chi(\Sigma) < 0\) we have by \cite[Section 2.4]{MR3382024} an inclusion $i\colon B_m(\Sigma) \to \mcg(\Sigma; m)$ where $\mcg(\Sigma;m)$ denotes the
mapping class group of $\Sigma$ with $m$ punctures. In what follows we identify $B_m(\Sigma)$ with $i(B_m(\Sigma))$.

Let $L$ be an infinite virtually abelian subgroup  of \(\pb_{m}(\Sigma)\). By \cite[Proposition 4.8]{MR3767218}\cite[Proposition 5.9]{Nucinkis:Petrosyan}\cite[Theorem 4.10] {Rita:Porfirio:Luis} there is finite index subgroup $K$ of $L$ such that $N_{\mcg(\Sigma,m)}[L]=N_{\mcg(\Sigma,m)}(K)$. The claim follows from \cref{property:pass:subgroup}.


Now suppose that $\chi(\Sigma) = 0$, so $\Sigma$ is the torus $T$, the
Klein bottle $K$, the cylinder $\Sigma_{0}^2$, or the Möbius band $MB$.
The cylinder $\Sigma_{0}^2$ embeds in $T$ as a subsurface satisfying the
hypotheses of \cref{lemma:subsurface-injective} (the complement is an
annulus), and similarly $MB$ embeds in $K$ with complement a Möbius band.
Hence $B_n(\Sigma_{0}^2)$ injects into $B_n(T)$ and $B_n(MB)$ injects
into $B_n(K)$ by \cref{lemma:subsurface-injective}. Moreover, $B_n(K)$
embeds as a subgroup of $B_{2n}(T)$ (see, for example,
\cite{MR2927811}). Therefore, by \cref{property:pass:subgroup}, it
suffices to prove the claim for $B_n(T)$.

If $n=1$ then $B_1(T)=\pi_1(T)=\Z^2$ and the claim is clear. For $n \geq 2$, there is an exact sequence \cite{FM12}
\[
  1 \to Z(B_n(T)) \to B_n(T) \to \mcg(T;n) \to \mcg(T) \to 1.
\]
In particular, $B_n(T)/Z(B_n(T))$ embeds as a subgroup of $\mcg(T;n)$. By
\cite[Theorem 4.10]{Rita:Porfirio:Luis}, every virtually abelian subgroup $H \leq \mcg(T;n)$ admits a
finite index subgroup $H' \leq H$ such that
$N_{\mcg(T;n)}[H] = N_{\mcg(T;n)}(H')$; hence \cite[Lemma 4.16]{Rita:Porfirio:Luis}
implies that $B_n(T)/Z(B_n(T))$ satisfies the same property.

Finally, consider the central short exact sequence
\[
  1 \to Z(B_n(T)) \to B_n(T) \to B_n(T)/Z(B_n(T)) \to 1.
\]
Since $Z(B_n(T)) \cong \mathbb{Z}^2$ \cite[Proposition 4.2]{Paris:Rolfsen}, an application of \cite[Proposition 4.15]{Rita:Porfirio:Luis}
yields that $B_n(T)$ itself satisfies the desired property.

\end{proof}

\begin{lemma}\label{exten:gd:inequality}\cite[Theorem 5.16]{Lu05}
Let 
\[
1 \to K \to G \to H \to 1
\]
be a short exact sequence of groups. Suppose that for any group \( M \) which contains a finite-index subgroup isomorphic to \( K \), we have \( \gdfin(M) = \gdfin(K) \). Then the following inequality holds
\[
\gdfin(G) \leq \gdfin(K) + \gdfin(H).
\]
\end{lemma}

\begin{lemma}\label{upper:bound:weil}
Let \( \Sigma \) be a connected surface of finite type with Euler characteristic \( \chi(\Sigma) \leq 0 \). For every infinite virtually abelian subgroup \( L \) of \( \pb_{m}(\Sigma) \), we have 
\[
\gdfin(W_G(L)) \leq \cd(\pb_{m}(\Sigma)) - \rank(L).
\]
\end{lemma}

\begin{proof}
We prove the result by induction on \( m \), the number of strands. When \( m = 1 \), we have \( \pb_1(\Sigma) = \pi_1(\Sigma) \), the fundamental group of the surface \( \Sigma \). If \( \chi(\Sigma) < 0 \), it follows that \( \pi_1(\Sigma) \) is torsion-free and hyperbolic. Therefore, any infinite virtually abelian subgroup \( L \) of \( \pi_1(\Sigma) \) must be isomorphic to \( \mathbb{Z} \). In a hyperbolic group, we know that \( C_G(L) \cong \mathbb{Z} \). Moreover, since \( C_G(L) \) has finite index in \( N_G(L) \), it follows that \( N_G(L) \cong \mathbb{Z} \) as well. Therefore, we have
\[
\gdfin(W_G(L)) \leq \cd(\pb_1(\Sigma)) - \rank(L),
\]

If $\chi(\Sigma) = 0$, then $\Sigma$ is the torus $T$, the Klein bottle $K$, 
the cylinder $\Sigma_0^2$, or the Möbius band $MB$. Recall that the fundamental groups of these surfaces are respectively $\mathbb{Z}^2$, $\mathbb{Z} \rtimes \mathbb{Z}$, $\mathbb{Z}$ and $\mathbb{Z}$. 
The claim follows by a direct verification in each case. For instance, we consider the Klein bottle (the rest of the cases are similar). Here 
$\pi_1(\Sigma) \cong \mathbb{Z} \rtimes \mathbb{Z}$, every subgroup is either 
virtually $\mathbb{Z}$ or virtually $\mathbb{Z}^2$.

\begin{itemize}
    \item If $H$ is virtually $\mathbb{Z}$, 
then $N(H)$ is either infinite cyclic or virtually $\mathbb{Z}^2$. In the former 
case $W(H) = N(H)/H$ is finite, so $\gdfin(W(H)) = 0$. In the latter case 
$W(H)$ is virtually $\mathbb{Z}$, and
\[
\gdfin(W(H)) = 1 \leq \cd(N(H)) - 1 = 2 - 1 = 1.
\]

\item If $H$ is virtually $\Z^2$ we have that $N(H)$ is virtually $\Z^2$, then $W(H)$ is virtually $\Z$, and
\[
\gdfin(W(H)) = 1 \leq \cd(N(H)) - 1 = 2 - 1 = 1.
\]
\end{itemize}
This completes the base case. 

Assume the lemma holds for \( m \). We now prove it for \( m + 1 \). By \cref{Fadell:Neuwirth} we have the following short exact sequence
\[
1 \to \pi_1(\Sigma - Q_m) \to \pb_{m+1}(\Sigma) \to \pb_m(\Sigma) \to 1,
\]
where \( Q_m \) denotes the set of \( m \) distinct points on \( \Sigma \).

Restricting this sequence to \( N_G(L) \), we obtain the following exact sequence:
\[
1 \to F \cap N_G(L) \to N_G(L) \to Q\subseteq N_{\pb_m(\Sigma)}(\varphi(L)) ,
\]
where \( F = \pi_1(\Sigma - Q_m) \), and \( \varphi \) is the projection map \( \pb_{m+1}(\Sigma) \to \pb_m(\Sigma) \).

We now analyze two cases based on whether \( \varphi(L) = 0 \) or \( \varphi(L) \neq 0 \):

\noindent \textbf{Case 1:} If \( \varphi(L) = 0 \), then \( L \cong \mathbb{Z} \), and we have the exact sequence:
\[
1 \to N_F(L) \to N_G(L) \to \pb_m(\Sigma) \to 1.
\]
Since \( W_F(L) \) is finite, and any finite extension of a finite group has \( \gdfin \) equal to zero, it follows from \cref{exten:gd:inequality} that
\[
\gdfin(W_G(L)) \leq \cd(\pb_m(\Sigma)) \leq \cd(\pb_{m+1}(\Sigma)) - \rank(L).
\]

\noindent \textbf{Case 2:} If \( \varphi(L) \neq 0 \), we further divide into two subcases depending on the rank of \( L \):

\noindent \textbf{Subcase 2.1:} If \( \rank(L) = 1 \), then \( L \cong \mathbb{Z} \), and we have the exact sequence:
\[
1 \to F \cap N_G(L) \to N_G(L) \to Q\subseteq N_{\pb_m(\Sigma)}(\varphi(L)) .
\]

Since \( F \cap N_G(L)\) is free, and any finite extension of a free group has \( \gdfin \) equal 1, it follows from \cref{exten:gd:inequality} that
\[
\begin{aligned}
\gdfin(W_G(L)) &\leq 1 + \gdfin(Q/\varphi(L)) \\
&\leq 1 + \gdfin(N_{\pb_m(\Sigma)}(\varphi(L))/\varphi(L)) \\
&\leq 1 + \cd(\pb_m(\Sigma)) - \rank(L) \\
&= \cd(\pb_{m+1}(\Sigma)) - \rank(L).
\end{aligned}
\]

the third inequality holds by the induction hypothesis and the last inequality by \cref{vcd:surfaces:braid:groups:0}.

\noindent \textbf{Subcase 2.2:} If \( \rank(L) \geq 2 \), then \( N_F(L\cap F ) \cong \mathbb{Z} \), it follows that \( W_F(L\cap F) \) is finite, and any finite extension of finite group has \( \gdfin = 0 \). Thus, by \cref{exten:gd:inequality}, we have

\[
\begin{aligned}
\gdfin(W_G(L)) &\leq \gdfin(N_{\pb_m(\Sigma)}(\varphi(L))/\varphi(L)) \\
&\leq \cd(\pb_m(\Sigma)) - \rank(\varphi(L)) \\
&= \cd(\pb_m(\Sigma)) - \rank(L)+1 \\
&= \cd(\pb_{m+1}(\Sigma)) - \rank(L).
\end{aligned}
\]
where, again, the second inequality follows from the induction hypothesis and the last inequality by \cref{vcd:surfaces:braid:groups:0}.

This completes the induction step, and thus the proof of the lemma.
\end{proof}

\begin{lemma}\label{caso:base:dimension}
Let $\Sigma_{g,k}^r$ be a surface with $\chi(\Sigma_{g,k}^r) \leq 0$. Then
\[
  \gd(\pb_n(\Sigma_{g,k}^r)) = \cd(\pb_n(\Sigma_{g,k}^r)).
\]
\end{lemma}

\begin{proof}
We consider two cases according to whether $r + k = 0$ or not.

\medskip
\noindent\textbf{Case 1: $r + k = 0$.}
In this case $\Sigma_{g,k}^r$ is a closed surface, and by \cref{vcd:surfaces:braid:groups:0},
\[
  \cd(\pb_n(\Sigma_{g,k}^r)) = n + 1.
\]

\noindent\textit{Subcase $n = 1$.}
We have $\pb_1(\Sigma_{g,k}^r) = \pi_1(\Sigma_{g,k}^r)$, and $\pi_1(\Sigma_{g,k}^r)$
acts freely and properly on $\mathbb{H}^2$ or $\mathbb{R}^2$, which are
models for $E\pi_1(\Sigma_{g,k}^r)$ of dimension~$2$. Hence
\[
  2 = \cd(\pb_1(\Sigma_{g,k}^r)) \leq \gd(\pb_1(\Sigma_{g,k}^r)) \leq 2,
\]
and therefore $\gd(\pb_1(\Sigma_{g,k}^r)) = \cd(\pb_1(\Sigma_{g,k}^r))$.

\noindent\textit{Subcase $n \geq 2$.}
We have $\cd(\pb_n(\Sigma_{g,k}^r)) \geq 3$, and the result follows from
\[
  \cd(\pb_n(\Sigma_{g,k}^r)) \leq \gd(\pb_n(\Sigma_{g,k}^r))
  \leq \max\{3,\, \cd(\pb_n(\Sigma_{g,k}^r))\} = \cd(\pb_n(\Sigma_{g,k}^r)).
\]

\medskip
\noindent\textbf{Case 2: $r + k \neq 0$.}
By \cref{vcd:surfaces:braid:groups:0},
\[
  \cd(\pb_n(\Sigma_{g,k}^r)) = n.
\]
Since $\cd(\pb_n(\Sigma_{g,k}^r)) \leq \gd(\pb_n(\Sigma_{g,k}^r))$ always holds,
it suffices to show that $\gd(\pb_n(\Sigma_{g,k}^r)) \leq n$.
We proceed by induction on $n$. If $n = 1$ then $\pb_1(\Sigma_{g,k}^r) = \pi_1(\Sigma_{g,k}^r)$, which is a free group,
and it is well known that $\gd(\pi_1(\Sigma_{g,k}^r)) = 1$.

Assume the result holds for $n$; we prove it for $n+1$.
The Fadell--Neuwirth fibration yields a short exact sequence
\[
  1 \to \pi_1(\Sigma_{g,k+r+n}) \to \pb_{n+1}(\Sigma_{g,k+r})
    \to \pb_n(\Sigma_{g,k+r}) \to 1.
\]
By \cref{exten:gd:inequality} and the inductive hypothesis,
\[
  \gd(\pb_{n+1}(\Sigma_{g,k+r}))
  \leq \gd(\pi_1(\Sigma_{g,k+r+n})) + \gd(\pb_n(\Sigma_{g,k+r}))
  \leq 1 + n = n + 1.
\]
This completes the proof.
\end{proof}
 We are now in a position to prove the main theorem.

\begin{theorem}[Upper bound]\label{upper:bound:gd:fn}
Let $\Sigma$ be a connected surface of finite type, possibly non-orientable
and possibly with punctures, with $\chi(\Sigma) \leq 0$.
For every integer $r \geq 0$,
\[
  \gd_{\calF_r}(\pb_n(\Sigma)) \leq \cd(\pb_n(\Sigma)) + r.
\]
\end{theorem}

\begin{proof}
The proof is by induction on \( r \). Let \( G = \pb_n(\Sigma) \). For \( r = 0 \), it follows from \cref{caso:base:dimension}
\[
\gd_{\calF_0}(G) = \gd(G) = \cd(G).
\]
Now, assume that the inequality holds for all \( r < k \). We now prove it for \( r = k \). Consider the equivalence relation \( \sim \) on \( \calF_k - \calF_{k-1} \) defined by commensurability, and let \( I \) be a complete set of representatives of the equivalence classes in \( (\calF_k - \calF_{k-1}) / \sim \). By \cref{lw:gd:upper:bound}, we have
\[
\gd_{\calF_k}(G) \leq \max \left\{ \gd_{\calF_{k-1}}(G) + 1, \gd_{\calF_k[L]}(N_G[L]) \mid L \in I \right\}.
\]
Using the inductive hypothesis, we can bound \( \gd_{\calF_{k-1}}(G) \leq \cd(G) + (k-1) \), and thus
\[
\gd_{\calF_k}(G) \leq \max \left\{ \cd(G) + k, \gd_{\calF_k[L]}(N_G[L]) \mid L \in I \right\}.
\]
To complete the proof, it suffices to show that \( \gd_{\calF_k[L]}(N_G[L]) \leq \cd(G) + k \) for all \( L \in I \).

Let \( L \in I \). By \cref{commen:equal:norma}, there exists a subgroup \( H \leq G \) commensurable with $L$ such that \( N_G[L] = N_G(H) \). Therefore, we may replace \( N_G[L] \) with \( N_G(H) \) and work with \( \calF_k[H] \).

We write the family 
\[
\calF_k[H] = \{ K \leq N_G(H) \mid K \in \calF_k - \calF_{k-1}, K \sim H \} \cup (\calF_{k-1}\cap N_G(H))
\]
as the union of two families, \( \calF_k[H] = \calG \cup (\calF_{k-1} \cap N_G(H))\), where \( \calG \) is generated by $$\{ K \leq N_G(H) \mid K \in \calF_k - \calF_{k-1}, [K] = [H] \}. $$ By \cref{upper:bound:gd:two:families}, we obtain
\[
\begin{split}
\gd_{\calF_k[L]}(N_G(H)) & \leq \max \left\{ \gd_{\calF_{k-1}}(N_G(H)), \gd_{\calG \cap (\calF_{k-1}\cap N_G(H))}(N_G(H)) + 1, \gd_{\calG}(N_G(H)) \right\} \\
& \leq \max \left\{ \cd(G) + k - 1, \gd_{\calG \cap (\calF_{k-1}\cap N_G(H))}(N_G(H)) + 1, \gd_{\calG}(N_G(H)) \right\},
\end{split}
\]
where the second inequality follows from the inductive hypothesis.

We now proceed to prove the following inequalities:
\begin{enumerate}[i)]
    \item \( \gd_{\calG}(N_G(H)) \leq \cd(G) - k \),
    \item \( \gd_{\calG \cap (\calF_{k-1}\cap N_G(H))}(N_G(H)) \leq \cd(G) + k - 1 \).
\end{enumerate}
Once these inequalities are established, we conclude that \( \gd_{\calF_k[L]}(N_G(H)) \leq \cd(G) + k \).

To prove item \( i) \), note that a model for \( E_{\calF_0}(W_G(H)) \) is a model for \( E_{\calG}(N_G(H)) \) via the projection \( N_G(H) \to W_G(H) \). Therefore, the claim follows directly from \cref{upper:bound:weil}.

For item \( ii) \), we apply \cref{nested:families:lw} to the inclusion \( \calG \cap (\calF_{k-1}\cap N_G(H)) \subset \calG \), we obtain
\[
\gd_{\calG \cap (\calF_{k-1}\cap N_G(H))}(N_G(H)) \leq \gd_{\calG}(N_G(H)) + d
\]
for some \( d \) such that for any \( K \in \calG \), we have \( \gd_{\calG \cap (\calF_{k-1}\cap N_G(H))}(K) \leq d \).

Since we have already established that \( \gd_{\calG}(N_G(H)) \leq \cd(G) - k \), our next task is to show that \( d \) can be chosen to be \( 2k - 1 \).

Recall that any \( K \in \calG \) is virtually \( \mathbb{Z}^t \) for some \( 0 \leq t \leq k \). We now consider two cases: If \( K \in \calG \) is virtually \( \mathbb{Z}^t \) for some \( 0 \leq t \leq k-1 \), then \( K \in \calF_{k-1} \), so \( K \in \calG \cap (\calF_{k-1}\cap N_G(H)) \) and \( \gd_{\calG \cap \calF_{k-1}}(K) = 0 \).
 If \( K \in \calG \) is virtually \( \mathbb{Z}^k \), we claim that \( (\calG \cap (\calF_{k-1}\cap N_G(H))) \cap K = \calF_{k-1} \cap K \). The inclusion \( (\calG \cap \calF_{k-1}) \cap K \subseteq \calF_{k-1} \cap K \) is clear. For the reverse inclusion, let \( M \in \calF_{k-1} \cap K \). Since \( K \leq N_G(H)\), we have \( \calF_{k-1} \cap K \subseteq \calF_{k-1} \cap N_G(H)\), and thus \( M \in \calF_{k-1} \cap N_G(H) \). Moreover, \( M \leq K \in \calG \), so \( M \in (\calG \cap (\calF_{k-1}\cap N_G(H))) \cap K \). Therefore, we conclude
\[
\gd_{(\calG \cap \calF_{k-1}) \cap K}(K) = \gd_{\calF_{k-1} \cap K}(K) \leq k + k - 1 = 2k - 1,
\]
where the last inequality follows from \cite[Proposition 1.3]{tomasz}.
\end{proof}
We now prove the lower bound. This requires a few preliminary lemmas.

\begin{lemma}\label{free:abelian:subgroup:cd}
$\pb_n(\Sigma_0^{2})$ has a free abelian subgroup of rank $n$. 
\end{lemma}
\begin{proof}
\begin{figure}[!ht]
    \centering
    \includegraphics[width=0.5\linewidth]{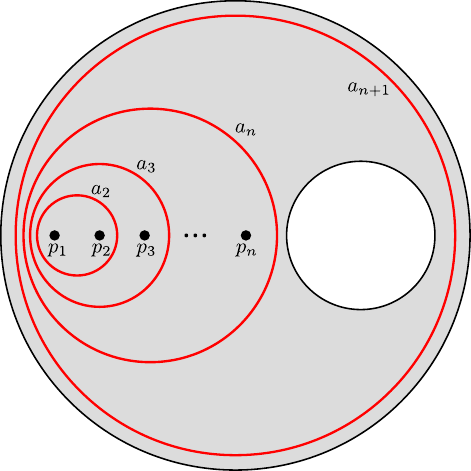}
    \caption{$\Sigma_0^2$ with $n$ marked points and $n$ simple closed curves. 
    }
    \label{absubgroup}
\end{figure}

We present $\Sigma_0^2$ as in \cref{absubgroup}, with $n$ marked points $p_1,\ldots,p_n$, which are the base points of the braids in $\pb_n(\Sigma_0^2)$. Consider the simple closed curves represented in \cref{absubgroup}: for $i=2,\ldots,n$, the curve $a_i$ encloses the punctures $p_1,\ldots,p_i$. The curve $a_{n+1}$ encloses the punctures $p_1,\ldots,p_n$ and also a boundary component.

It is well known (see for example \cite[Section 5]{Crisp99}), that the braid group of an annulus $B_n(\Sigma_0^2)$, with $n$ strands, is isomorphic to the subgroup of the classical braid group of the disk $B_{n+1}(\Sigma_0^1)$ on $n+1$ strands, formed by those braids whose last strand is vertical (equivalently, those braids whose induced permutation sends the $(n+1)$-st marked point to itself). The isomorphism is just given by collapsing one of the boundary components to form a new marked point. 

If we consider, in \cref{absubgroup}, the boundary component enclosed by the curve $a_{n+1}$ as a new marked point ($p_{n+1}$), we can describe $n$ braids in the annulus, $t_2,\ldots,t_{n+1}$, as follows: $t_i$ will be the full twist in the braid group $ \pb_i$ of the punctured disk enclosed by $a_i$. In other words, the braid $t_i$ corresponds to the Dehn twist along the curve $a_i$. In terms of the Artin generators of $\pb_i$, we have:
$$
    t_i=(\sigma_1\cdots \sigma_{i-1})^i.
$$

Notice that each group $\mathcal \pb_i$ embeds naturally in $\pb_{i+1}$, so we have a chain of subgroups
$$
 \pb_2\subset \pb_3 \subset \cdots \subset \pb_{n+1},
$$
and it is well known \cite{Chow48} that $t_i$ generates the center of $\pb_i$, for $i=2,\ldots,n+1$.

We will see that $\langle t_2,\ldots,t_{n+1}\rangle$ is a free abelian subgroup of $\pb_n(\Sigma_0^2)$ of rank $n$.

First, it is obvious that each braid $t_i$ is a pure braid (it induces the trivial permutation on the marked points). Moreover, for every $i<j$, $t_i\in \pb_i\subset \pb_j$ and $t_j$ is central in $\pb_j$, hence $t_i$ and $t_j$ commute. One can also see that these braids commute by realizing that they correspond to Dehn twists along disjoint curves. Therefore, $\langle t_2,\ldots,t_{n+1}\rangle$ is an abelian subgroup of $\pb_n(\Sigma_0^2)$. It remains to show that $\{t_2,\ldots,t_{n+1}\}$ is a free basis of this abelian group.

Consider an element $x=t_2^{e_2}\cdots t_{n+1}^{e_{n+1}}\in \langle t_2,\ldots,t_{n+1} \rangle$, where $e_2,\ldots,e_{n+1}\in \mathbb Z$, and suppose that $x=1$. We need to show that $e_2=\cdots=e_{n+1}=0$. 

Consider the algebraic crossing number $c_{n,n+1}(x)$ between the strands $n$ and $n+1$ of the braid $x$ (recall that we are considering the boundary component enclosed by $a_{n+1}$ as a marked point). Since these two strands do not cross in $t_2,\ldots,t_n$, and they cross twice in $t_{n+1}$, it follows that $c_{n,n+1}(x)=2e_{n+1}$. But $x$ is the trivial braid, hence $c_{n,n+1}(x)=0$. Therefore $e_{n+1}=0$.

Let now $i<n+1$, and assume that $e_{i+1}=e_{i+2}=\cdots=e_{n+1}=0$. This assumption implies that $x=t_2^{e_2}\cdots t_i^{e_i}=1$. Looking at the crossing number $c_{i-1,i}(x)$, we have that it is equal to $2e_i$ (since the strands $i-1$ and $i$ do not cross in $t_2,\ldots,t_{i-1}$ and they cross twice in $t_i$), and at the same time it is equal to $0$. Therefore $e_i=0$.

Applying the same argument for $i=n+1,n,\ldots,2$, we finally obtain $e_2=\cdots=e_{n+1}=0$, as we wanted to show.
\end{proof}

\begin{lemma}\cite[Theorem~1.1]{PL}\label{Virtually:dimension:Z}
Let $k,n\in \mathbb{N}$ such that $0\leq k< n$. Let $G$ be a virtually $\Z^n$ group. Then  $\gd_{\calF_k}(G)=\cd_{\calF_k}(G)=n+k$.
\end{lemma}

\begin{theorem}
Let \( \Sigma \) be a connected surface of finite type with at least one puncture or one boundary component and negative Euler characteristic. Then, 
\[
\gd_{\calF_r}(\pb_n(\Sigma)) = \cd(\pb_n(\Sigma)) + r.
\]
\end{theorem}
\begin{proof}
The upper bound was established in \cref{upper:bound:gd:fn}, and it remains to prove the lower bound. Observe that we can embed $\Sigma_{0,2}$ into $\Sigma$ as a subsurface whose complement is not a disc. By applying \cref{lemma:subsurface-injective}, we obtain an injective homomorphism
\[
\pb_n(\Sigma_0^{2}) \to \pb_n(\Sigma).
\] 
Consequently, we have the following inequalities:

\begin{alignat}{2}
\gd_{\calF_r}(\pb_n(\Sigma)) &\geq \gd_{\calF_r}(\pb_n(\Sigma_0^{2})) & & \\
&\geq \cd(\pb_n(\Sigma_0^{2})) + r &\quad&\text{by \cref{free:abelian:subgroup:cd} and \cref{Virtually:dimension:Z}} \\
&= \cd(\pb_n(\Sigma)) + r &\quad&\text{by \cref{vcd:surfaces:braid:groups:0}}.
\end{alignat}

\end{proof}

\begin{theorem}
Let \( \Sigma \) be a closed connected surface of finite type and negative Euler characteristic. Then, 
\[
\cd(\pb_n(\Sigma)) + r-1\leq\gd_{\calF_r}(\pb_n(\Sigma)) \leq \cd(\pb_n(\Sigma)) + r.
\]
\end{theorem}
\begin{proof}
The upper bound follows from \cref{upper:bound:gd:fn}. To establish the lower bound, observe that we can embed $\Sigma_{0,2}$ into $\Sigma$ as a subsurface whose complement is not a disc. By applying \cref{lemma:subsurface-injective}, this inclusion induces an injective homomorphism
\[
\pb_n(\Sigma_0^2) \hookrightarrow \pb_n(\Sigma).
\]
Hence, we obtain:
\[
\begin{split}
\gd_{\calF_r}(\pb_n(\Sigma)) &\geq \gd_{\calF_r}(\pb_n(\Sigma_0^2)) \\
&= \cd(\pb_n(\Sigma_0^2)) + r \\
&= \cd(\pb_n(\Sigma)) + r - 1 \text{ By \cref{vcd:surfaces:braid:groups:0}}.
\end{split}
\]
\end{proof}

\subsection{The case of the sphere}

\begin{lemma}\label{upper:bound:gdf_n:ses}
Let $1 \to F \to G \xrightarrow{p} H \to 1$ be a short exact sequence with $F$
finite. Then
\[
  \gd_{\calF_n}(G) \leq \gd_{\calF_n}(H)
  \quad \text{for all } n \geq 0.
\]
\end{lemma}

\begin{proof}
Let $\calG_n$ denote the family of virtually free abelian subgroups of $H$
of rank $r$, for some  $0 \leq r \leq n$.

The pull-back family $p^*(\calG_n)$, generated by
$\{p^{-1}(L) : L \in \calG_n\}$, contains the family $\calF_n$
of $G$. We may then apply \cref{nested:families:lw} to obtain
\[
  \gd_{\calF_n}(G)
  \leq \gd_{p^*(\calG_n)}(G)
  + \max\!\left\{\gd_{\calF_n}\!\left(p^{-1}(L)\right)
    : L \in \calG_n \right\}.
\]

Note that a model for $E_{\calG_n} H$ is a model for
$E_{p^*(\calG_n)} G$ via the action induced by the projection
$p$. On the other hand, $p^{-1}(L)$ is virtually free abelian of rank $r$ for
some $0 \leq r \leq n$, so
\[
  \gd_{\calF_n}\!\left(p^{-1}(L)\right) = 0.
\]

Therefore
\[
  \gd_{\calF_n}(G) \leq \gd_{\calF_n}(H). \qedhere
\]
\end{proof}

\begin{lemma}\label{free:abelian:subgroup:sphere}
Let $\Sigma$ be either the sphere $S^2$ or the projective plane $\mathbb{R}P^2$.
Then $\pb_n(\Sigma)$ contains a free abelian subgroup of rank $\vcd(B_n(\Sigma))$.
\end{lemma}
\begin{proof}
By \cref{vcd:surfaces:braid:groups:0}, we have $\vcd(B_n(S^2)) = n-3$ and
$\vcd(B_n(\mathbb{R}P^2)) = n-2$.

\textbf{Case $\Sigma = S^2$.}
By \cref{surfacebraid:puntures:boundaries}, $\pb_{n-3}(\Sigma_0^3) \cong \pb_{n-3}(\Sigma_{0,3})$,
and by \cite[Remark~8]{MR3382024}, $\pb_{n-3}(\Sigma_{0,3}) \leq \pb_n(S^2)$.
Since $\Sigma_{0}^3$ contains a subsurface homeomorphic to $\Sigma_{0}^2$ whose complement
is not a disc, then by \cref{lemma:subsurface-injective} $\pb_{n-3}(\Sigma_{0}^2)<\pb_{n-3}(\Sigma_{0}^3)$. Now \cref{free:abelian:subgroup:cd} applies, yielding a free abelian subgroup
of rank $n-3$ in $\pb_{n-3}(\Sigma_{0}^3)$, and hence in $\pb_n(S^2)$.

\textbf{Case $\Sigma = \mathbb{R}P^2$.}
By \cref{surfacebraid:puntures:boundaries}, $\pb_{n-2}(\Sigma_{-1}^2) \cong
\pb_{n-2}(\Sigma_{-1,2})$, and by \cite[Remark~8]{MR3382024},
$\pb_{n-2}(\Sigma_{-1,2}) \leq \pb_n(\mathbb{R}P^2)$.
The projective plane decomposes as a Möbius band glued to a disc along their boundary;
consequently, $\Sigma_{-1}^2$ is homeomorphic to a Möbius band with one open disc removed
(see \cref{fig:mobius-minus-disk}).
In particular, $\Sigma_{-1}^2$ contains a subsurface homeomorphic to $\Sigma_{0}^2$ whose
complement is not a disc, then by \cref{lemma:subsurface-injective} $\pb_{n-2}(\Sigma_{0}^2) <\pb_{n-2}(\Sigma_{-1}^2)$. Now \cref{free:abelian:subgroup:cd} applies and yields a free abelian
subgroup of rank $n-2$ in $\pb_{n-2}(\Sigma_{-1,2})$, and hence in $\pb_n(\mathbb{R}P^2)$.
\end{proof}
\begin{figure}[h]
\centering
\begin{tikzpicture}

\begin{tikzpicture}

\fill[purple!10] (0,0) rectangle (5,3);

\begin{scope}
  \clip (2.5,1.5) circle (0.8);
  \fill[white] (2.5,1.5) circle (0.8);
  \foreach \y in {0.7,0.9,...,2.3}{
    \draw[orange!40, thin] (1.7,\y) -- (3.3,\y+0.4);
  }
\end{scope}

\draw[orange!70, dashed, thick] (2.5,1.5) circle (0.8);

\draw[gray, thick] (0,0) -- (5,0);
\draw[gray, thick] (0,3) -- (5,3);

\draw[violet, very thick, -{Stealth}] (0,3) -- (0,1.5);
\draw[violet, very thick] (0,1.5) -- (0,0);

\draw[violet, very thick, -{Stealth}] (5,0) -- (5,1.5);
\draw[violet, very thick] (5,1.5) -- (5,3);

\node[violet] at (-0.3, 1.5) {$a$};
\node[violet] at (5.3, 1.5)  {$a$};

\end{tikzpicture}
\end{tikzpicture}
\caption{The surface $\Sigma_{-1}^2 \cong MB \setminus \operatorname{int}(D^2)$,
obtained from the Möbius band by removing an open disk.}
\label{fig:mobius-minus-disk}
\end{figure}
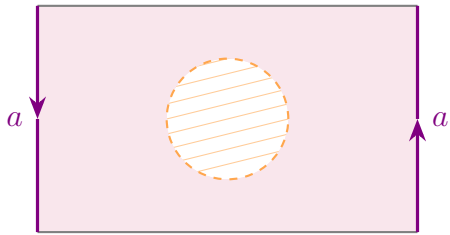

\begin{theorem}\label{abelian:dim:sphere}
Let $\Sigma$ be the sphere. For all integers $r\geq 0$
\[
  \gd_{\calF_r}(B_n(\Sigma)) = \vcd(B_n(\Sigma)) + r.
\]
\end{theorem}

\begin{proof}
If $n \leq 2$, then $B_n(S^2)$ is finite, so the claim is clear.

Now suppose  $n \geq 3$.
There is a short exact sequence \cite[Section~2.4]{MR3382024}
\[
  1 \to \Z_2 \to B_n(\Sigma) \to \mcg(\Sigma;n) \to 1.
\]
For the upper bound, we have

\begin{align*}
  \gd_{\calF_r}(B_n(\Sigma))
  &\leq \gd_{\calF_r}(\mcg(\Sigma;n))& \text{ By \cref{upper:bound:gdf_n:ses}} \\
  &\leq \vcd(\mcg(\Sigma;n)) + r &\text{ \cite[Theorem~1.5]{Rita:Porfirio:Luis}}\\
  &= \vcd(B_n(\Sigma)) + r,
\end{align*}

The lower bound follows from \cref{free:abelian:subgroup:sphere} and \cref{Virtually:dimension:Z}.
\end{proof}

\subsection{The case of the projective plane}

In this subsection, we aim to compute the $\calF_k$-geometric dimension of the pure braid group $\pb_n(\mathbb{R}P^2)$. 
As a first step, we establish that the proper geometric dimension equals their virtual cohomological dimension. 
Subsequently, we show that the commensurator of any virtually abelian subgroup of $\pb_n(\mathbb{R}P^2)$ can be realized as a normalizer. 
This result enables us to apply the strategy developed for the pure surface braid group with negative curvature.

\begin{lemma}\label{vcd:centra:finite:0:projec}
    For every finite subgroup $F$ of $\pmcg(\mathbb{R}P^2;n)$ we have $\vcd(W(F))=\vcd(N(F))=\vcd(C(F))=0$.
\end{lemma}
\begin{proof}
    First note that $W(F)$, $N(F)$ and $C(F)$ are commensurable, so that $\vcd(W(F))=\vcd(N(F))=\vcd(C(F))$. Hence it suffices to show that $\vcd(C(F))=0$. To this end, observe that $C(F)=\bigcap_{f\in F} C(f)$. We claim that $C(f)$ is finite for every nontrivial $f\in F$; granting this, $\vcd(C(F))\leq \vcd(C(f))=0$.

 To prove the claim, consider the following short exact sequence, which holds for $n\geq 2$  \cite[Section~2.4]{MR3382024}:
  \[
    1 \to \Z_2 \to B_n(\mathbb{R}P^2) \to \mcg(\mathbb{R}P^2;n) \to 1.
  \]
  The generator of the kernel is the full twist $\Delta^2$ \cite{MR3382024}, hence it is a pure braid. Therefore, we can restrict the above exact sequence to pure braids and mapping classes, obtaining:
  \[
    1 \to \Z_2 \stackrel{i}{\to} PB_n(\mathbb{R}P^2) \stackrel{p}{\to} \pmcg(\mathbb{R}P^2;n) \to 1.
  \]
  Let $f$ be a nontrivial finite order element in $\pmcg(\mathbb{R}P^2;n)$, and let $\alpha$ be such that $p(\alpha)=f$. Since $\ker(p)$ is finite, it follows that $\alpha$ has finite order. In \cite[Proposition~1]{MR2607582} it is shown that a maximal finite subgroup of $PB_n(\mathbb{R}P^2)$ is isomorphic to either $\mathbb{Z}_2$, or $\mathbb{Z}_4$, or $\mathcal Q_8$ (the latter only if $n=2$ or $n=3$). Therefore, the order of $\alpha$ is either 2 or 4. Moreover, it is shown in \cite[Proposition~15]{MR2100679} that the only element of order 2 in $PB_n(\mathbb{R}P^2)$ is precisely the full twist $\Delta^2$, which belongs to $\ker(p)$. Therefore, $\alpha$ cannot have order 2 (otherwise $f$ would be trivial). Hence, $\alpha$ has order 4, $\alpha^2=\Delta^2$, and $f$ has order 2.

  We recall also that, by \cite[Corollary~4]{MR2607582}, the centralizer of an element $\alpha\in PB_n(\mathbb{R}P^2)$ of order 4 is precisely $\langle \alpha\rangle$. Hence $\#(C(\alpha))=\#(\{1,\alpha,\alpha^2,\alpha^3\})=4$.

  Let $g\in C(f)$, and let $\gamma$ be such that $p(\gamma)=g$. We have $gfg^{-1}f^{-1}=1$, hence $\gamma \alpha \gamma^{-1} \alpha^{-1}$ is equal to either 1 or $\Delta^2$. In the former case, $\gamma\in C(\alpha)$. In the latter case, $\gamma \alpha \gamma^{-1} \alpha^{-1}=\alpha^2$, that is, $\gamma \alpha \gamma^{-1} = \alpha^{3}$. If the latter case occurs, fix one element $x\in PB_n(\mathbb{R}P^2)$ such that $x\alpha x^{-1}=\alpha^3$. Then $(x^{-1}\gamma)\alpha (\gamma^{-1}x) = x^{-1}\alpha^3x = \alpha$, that is, $x^{-1}\gamma\in C(\alpha)$, which implies that $\gamma$ belongs to the lateral class $xC(\alpha)$. We have then shown that $\gamma \in C(\alpha)\cup (xC(\alpha))=\{1,\alpha,\alpha^2,\alpha^3, x,x\alpha, x\alpha^2, x\alpha^3\}$. Now, since $p(\alpha^2)=1$, we obtain that $g=p(\gamma)\in \{1,f,p(x),p(x)f\}$. Therefore $\#(C(f))\leq 4$. 
\end{proof}

\begin{lemma}\label{finite:sungroup:ord:ateast2}
    Every finite subgroup $F$ of $\pmcg(\mathbb{R}P^2;n)$ and $n\geq 4$ has order at most 2.
\end{lemma}
\begin{proof}
It is a consequence of the last short exact sequence of \cite[Section~2.4]{MR3382024} and Proposition 1 of \cite{MR2607582}.  
\end{proof}

\begin{proposition}
    For the projective plane we have $\cd_{\calF_0}(\pmcg(\mathbb{R}P^2;n))=\vcd(\pmcg(\mathbb{R}P^2))$.
\end{proposition}
\begin{proof}
    For $n=1,2$ the group $\pmcg(\mathbb{R}P^2;n)$ is finite, and for $n=3$ it is virtually free; for $n=1,2$ the statement is immediate, for a virtually free group the fact is well known. It therefore suffices to establish the claim for $n\ge 4$.
    By \cite[Theorem 3.3]{AMP14} it is enough to show that every finite subgroup $F$ of $\pmcg(\mathbb{R}P^2;n)$ satisfies the inequality
    \[
        \vcd(W(F))+\lambda(F)\leq \vcd(\pmcg(\mathbb{R}P^2;n)).
    \]
    By \cref{vcd:centra:finite:0:projec}, every finite subgroup $F$ of $\pmcg(\mathbb{R}P^2;n)$ satisfies $\vcd(W(F))=0$, and by \cref{finite:sungroup:ord:ateast2} every such subgroup has order at most $2$, hence $\lambda(F)\leq 2$. Since $\vcd(\pmcg(\mathbb{R}P^2;n))=n-2$ for all $n\ge 3$, the desired inequality holds.
\end{proof}

\begin{corollary}\label{vcd:gdfin:pure:projec}
    For the projective plane, we have $\gd_{\calF_0}(\pb_n(\mathbb{RP}^2)) = \vcd(\pb_n(\mathbb{RP}^2))$.
\end{corollary}

\begin{proof}
    For $n = 1, 2$, the group $\pmcg(\mathbb{RP}^2; n)$ is finite, and for $n = 3$ it is virtually free. 
    In all these cases, the statement follows immediately. 

    If $n=4$ we use the following short exact sequence (\cite{Fadell1962}, see also \cite{GG2010}):

    \[
1 \longrightarrow \pi_1(\mathbb{R}P^2-Q) \longrightarrow \pb_{4}(\mathbb{R}P^2) \xrightarrow{\;p\;} \pb_3(\mathbb{R}P^2) \longrightarrow 1, \tag{$*$}
\]
where $Q\subset\mathbb{R}P^2$ is a set of $3$ points and $\pi_1(\mathbb{R}P^2-Q)$ is free. Since every virtually free group has $\gdfin=1$, we can use \cref{exten:gd:inequality} deduce that 
$$\gdfin(\pb_{4}(\mathbb{R}P^2))\leq \gdfin(\pi_1(\mathbb{R}P^2-Q))+ \gdfin(\pb_3(\mathbb{R}P^2))=1+1=2.$$
Since $\pb_{4}(\mathbb{R}P^2)$ is virtually torsion-free, we have $\gdfin(\pb_{4}(\mathbb{R}P^2))\geq \vcd(\pb_{4}(\mathbb{R}P^2))=2$.

    For $n \geq 5$, by \cref{free:abelian:subgroup:sphere}, the group $\pmcg(\mathbb{RP}^2; n)$ admits a free abelian subgroup of rank $n-2$. 
    Therefore, the monotonicity of cohomological dimension yields
    \[
    \cd_{\calF_0}(\pmcg(\mathbb{RP}^2; n)) \geq n - 2.
    \]
    The result now follows from the inequalities \cite[Theorem 0.1]{LM00}:
    \[
    \cd_{\calF_0}(G) \leq \gd_{\calF_0}(G) \leq \max\{\cd_{\calF_0}(G), 3\}.
    \]
\end{proof}

We now turn to showing that the commensurators of virtually abelian subgroups in the pure braid group of the projective plane $\pmcg(\mathbb{RP}^2; n)$ is realized as normalizers.

\begin{definition}[(Strong) Condition $C$]
Let $n$ be a natural number. We say that a group $G$ satisfies condition
$C_n$ (resp.\ strong condition $C_n$) if for every free abelian subgroup
$H$ of $G$ of rank $n$, and for all finitely generated
$K \subset N_G[H]$, there is $H'$ commensurable with $H$
(resp.\ $H'$ a finite index subgroup of $H$) such that
\[
\langle H,K\rangle \subset N_G(H').
\]
Whenever $G$ satisfies condition $C_n$
(resp.\ strong condition $C_n$), for all $n$, we say that
$G$ satisfies condition $C$ (resp.\ strong condition $C$).
\end{definition}

\begin{lemma}\label{comm:equal:normalizer}
Let $G$ be a group and let $H$ be a free abelian subgroup.
Assume that $G$ satisfies condition $C$
(resp.\ strong condition $C$), and that $N_G[H]$
is finitely generated. Then there is $H'$
commensurable with $H$
(resp.\ $H'$ a finite index subgroup of $H$)
such that
\[
N_G[H]=N_G(H').
\]
\end{lemma}

To verify that braid groups of the projective plane satisfy condition C, we will show that all virtually abelian subgroups are separable. Recall that if $G$ is a group and $H$ a subgroup of $G$, then $H$ is called
\emph{separable} in $G$ if for every $g \in G \setminus H$, there exists a subgroup $K$
of finite index in $G$ such that $H \subset K$ but $g \notin K$.  It is well known that a subgroup is separable if and only if it is closed in the profinite topology. Recall that the profinite topology has as a base the family of right cosets $\{gK \mid K \leq G, [G;K]< \infty, \, g \in G\}$, where $K$ ranges over all subgroups of finite index in $G$. 

The following lemma is an immediate consequence of the definition.

\begin{lemma}\label{lem:finite-index-separability}
Let $K \leq G$ be groups with $[G:K] < \infty$, and let $H \leq K$.
Then $H$ is separable in $K$ if and only if $H$ is separable in $G$.
\end{lemma}

\begin{lemma}\label{lem:vab-finite-index}
Let $K \leq G$ be groups with $[G:K] < \infty$. If every virtually
abelian subgroup of $K$ is separable in $K$, then every virtually
abelian subgroup of $G$ is separable in $G$.
\end{lemma}

\begin{proof}
Let $H \leq G$ be a virtually abelian subgroup, and set $H_0 := H \cap K$.
Since $[G:K] < \infty$, we have $[H:H_0] \leq [G:K] < \infty$, so $H_0$
has finite index in $H$. In particular, $H_0$ is a virtually abelian
subgroup of $K$, so by hypothesis $H_0$ is separable in $K$.
By \cref{lem:finite-index-separability}, $H_0$ is separable in $G$.

Write $H = \bigsqcup_{i=1}^{r} h_i H_0$ for some $h_1, \ldots, h_r \in H$.
Since $H_0$ is separable in $G$, it is closed in the profinite topology
of $G$. Left translation by $h_i$ is a homeomorphism in this topology,
so each coset $h_i H_0$ is closed. Therefore $H$, as a finite union of
closed sets, is closed in the profinite topology of $G$, i.e., $H$ is
separable in $G$.
\end{proof}

\begin{corollary}\label{separable:sphere:braid}
    Every virtually abelian subgroup of $B_n(S^2)$ is separable.
\end{corollary}
\begin{proof}
Consider the short exact sequence
\[
  1 \to \Z_2 \to B_n(S^2) \xrightarrow{p} \mcg(S^2;n) \to 1.
\]
Since $B_n(S^2)$ is virtually torsion-free, it has a torsion-free
subgroup $H$ of finite index, then $p|_H \colon H \xrightarrow{\;\sim\;} p(H)$ is an isomorphism.
By \cite[Theorem 1.1]{LM07}, every virtually abelian subgroup of $\mcg(S^2;n)$
is separable in $\mcg(S^2;n)$; since $p(H)$ has finite index in $ \mcg(S^2;n)$, it follows from \cref{lem:finite-index-separability} that every virtually abelian subgroup of $p(H)$ is separable in $p(H)$. As $H \cong p(H)$, every virtually abelian subgroup of
$H$ is separable in $H$. Since $[B_n(S^2):H] < \infty$,
\cref{lem:vab-finite-index} implies that every virtually abelian
subgroup of $B_n(S^2)$ is separable in $B_n(S^2)$.
\end{proof}

\begin{proposition}\label{sphere:braid:conditionC}
  $B_n(S^2)$ satisfies condition C.
\end{proposition}

\begin{proof}
Let $H$ be a virtually free abelian subgroup of $B_n(S^2)$ of rank $n$, and let $K\subseteq N[H]$ be a finitely generated subgroup. By \cref{separable:sphere:braid}, the subgroup $H$ is separable in $B_n(S^2)$. Now, by \cite[Corollary~9]{CKRW20}, there exists a finite-index subgroup
\(H'\leq H\) such that \(H'\) is normal in \(\langle H,K\rangle\). Therefore,
\[
\langle H,K\rangle \subseteq N_{B_n(S^2)}(H').
\]
\end{proof}

\begin{lemma}\label{abelinan:subgroup:sphere:comm:fg}
Let \(H\) be a virtually abelian subgroup of \(B_n(S^2)\). Then the commensurator \(N[H]\) is finitely generated.
\end{lemma}

\begin{proof}
Consider the canonical short exact sequence
\[
1\longrightarrow \langle \Delta^2\rangle\cong \mathbb Z_2
\longrightarrow B_n(S^2)
\overset{\pi}{\longrightarrow}
\operatorname{Mod}(S^2,n)
\longrightarrow 1.
\]

Let
\[
\overline H:=\pi(H)\leq \operatorname{Mod}(S^2,n).
\]
Since \(\overline H\) is a quotient of the virtually abelian group \(H\), it is itself virtually abelian. Moreover,
\[
H\leq \pi^{-1}(\overline H),
\]
and, because \(\ker(\pi)=\langle \Delta^2\rangle\cong \mathbb Z_2\),
\[
[\pi^{-1}(\overline H):H]\leq 2.
\]
Hence \(H\) and \(\pi^{-1}(\overline H)\) are commensurable. Since commensurators depend only on the commensurability class, we obtain
\[
N[H]=N[\pi^{-1}(\overline H)].
\]

We claim that
\[
\pi\bigl(N[\pi^{-1}(\overline H)]\bigr)
=
N[\overline H].
\]

The inclusion
\[
\pi\bigl(N[\pi^{-1}(\overline H)]\bigr)
\subseteq
N[\overline H]
\]
is immediate.

For the reverse inclusion, let \(x\in N[\overline H]\), and choose
\(g\in B_n(S^2)\) such that \(\pi(g)=x\). Since
\[
[x\overline Hx^{-1}:x\overline Hx^{-1}\cap \overline H]<\infty
\]
and
\[
[\overline H:x\overline Hx^{-1}\cap \overline H]<\infty,
\]
it follows that
\[
[g\pi^{-1}(\overline H)g^{-1}:
g\pi^{-1}(\overline H)g^{-1}\cap \pi^{-1}(\overline H)]<\infty
\]
and
\[
[\pi^{-1}(\overline H):
g\pi^{-1}(\overline H)g^{-1}\cap \pi^{-1}(\overline H)]<\infty.
\]
Therefore
\[
g\in N[\pi^{-1}(\overline H)],
\]
and consequently
\[
x=\pi(g)\in
\pi\bigl(N[\pi^{-1}(\overline H)]\bigr).
\]
This proves the claim.

Set
\[
Q:=N[\overline H].
\]
The commensurator of every virtually abelian subgroup of rank $\geq 1$ in $\mcg(S^2,n)$ is finitely generated. For $n=1,2,3$ this is clear since $\mcg(S^2,n)$ is finite. For $n\geq 4$, the case of rank $\geq 2$ is \cite[Corollary~4.6]{Rita:Porfirio:Luis}, and the case of rank $1$ follows from \cite[Proposition~4.2]{Rita:Porfirio:Luis}, which identifies the commensurator with the normalizer, and \cite[Proposition~4.12]{MR3767218}, which gives finite generation of the centralizer; the latter suffices since the centralizer of a cyclic subgroup has finite index in its normalizer.

Hence \(Q\) is finitely generated. Since
\[
1\longrightarrow \langle \Delta^2\rangle
\longrightarrow N[H]
\overset{\pi}{\longrightarrow}
Q
\longrightarrow 1
\]
is exact, and \(\langle \Delta^2\rangle\) is finite, it follows that \(N[H]\) is finitely generated.
\end{proof}

\begin{proposition}\label{comm:realize:normalizer:pure:projec}
Let $\Sigma$ be either the sphere or the projective plane. For every virtually abelian subgroup $H$ of $B_n(\Sigma)$, there exists a subgroup $H'$ commensurable with $H$ such that
\[
N[H]=N[H']=N(H').
\]
\end{proposition}

\begin{proof}
There is an injective homomorphism (see, for example,
\cite{MR2927811})
\[
B_n(\mathbb{R}P^2)\longrightarrow B_{2n}(S^2).
\]
Hence, by \cref{property:pass:subgroup}, it is enough to prove the claim for $B_n(S^2)$. By \cref{comm:equal:normalizer}, it suffices to show that $B_n(S^2)$ satisfies condition C and that the commensurator of every virtually abelian subgroup of $B_n(S^2)$ is finitely generated. These two facts were established in \cref{sphere:braid:conditionC} and \cref{abelinan:subgroup:sphere:comm:fg}, respectively.
\end{proof}

\begin{proposition}\label{proper:dim:weyl:pure:projec}
Let $n\ge 1$ and let $H<PB_n(\mathbb{R}P^2)$ be a virtually abelian subgroup of rank $r=\rank(H)\ge 1$. Then
\[
\gdfin(W(H)) \;\le\; \vcd\big(PB_n(\mathbb{R}P^2)\big)+r.
\]
\end{proposition}

\begin{proof}
We proceed by induction on $n$.

\medskip
\noindent\textbf{Base cases $n=1,2$.}
The group $PB_n(\mathbb{R}P^2)$ is finite, so $W(H)$ is finite and $\vcd(PB_n(\mathbb{R}P^2))=0$. Then $\gdfin(W(H))=0\le 0+r$, and the claim is clear.

\medskip
\noindent\textbf{Base case $n=3$.}
The group $PB_3(\mathbb{R} P^2)$ is virtually free, so $\vcd(PB_3(\mathbb{R}P^2))=1$. Since every virtually abelian subgroup of a virtually free group is virtually cyclic, $N(H)$ is virtually $\mathbb{Z}$ (or finite), and hence $W(H)=N(H)/H$ is finite. Then $\gdfin(W(H))=0\le 1+r$, and the claim follows.

\medskip
\noindent\textbf{Inductive step.}
Let $n\ge 3$ and assume the result holds for $PB_n(\mathbb{R}P^2)$. Let $H<PB_{n+1}(\mathbb{R}P^2)$ be virtually abelian of rank $r=\rank(H)\geq1$. Consider the short exact sequence of Fadell-Neuwirth (\cite{Fadell1962}, see also \cite{GG2010}):
\[
1 \longrightarrow \pi_1(\mathbb{R}P^2-Q) \longrightarrow PB_{n+1}(\mathbb{R}P^2) \xrightarrow{\;p\;} PB_n(\mathbb{R}P^2) \longrightarrow 1, \tag{$*$}
\]
where $Q\subset\mathbb{R}P^2$ is a set of $n$ points and $\pi_1(\mathbb{R}P^2-Q)$ is free.

We distinguish two cases according to whether $p(H)$ is finite or infinite.

\medskip
\noindent\emph{Case 1: $p(H)$ is finite.}
Restricting $(*)$ to $N(H)$ we obtain
\[
1\to \pi_1(\mathbb{R}P^2-Q)\cap N(H) \to N(H) \xrightarrow{\;p\;} L \to 1, \qquad L:=p(N(H))\le N(p(H)).
\]

Passing to the quotient by $H$ yields
\begin{equation}\label{s.e.s:254}
1\to \dfrac{\pi_1(\mathbb{R}P^2-Q)\cap N(H)}{\pi_1(\mathbb{R}P^2-Q)\cap H} \to W(H) \to \dfrac{L}{p(H)} \to 1,
\end{equation}

Set $F=\pi_1(\mathbb{R}P^2-Q)$. We claim that the kernel
\[
K:=\dfrac{F\cap N(H)}{F\cap H}
\]
is finite. Since $p(H)$ is finite, $[H:F\cap H]=|p(H)|<\infty$; and as $H$ is infinite (being virtually abelian of rank $r\ge 1$), so is $F\cap H$. Now $F\cap H$ is free, being a subgroup of the free group $F$, and virtually abelian, being a subgroup of $H$; an infinite virtually abelian free group is infinite cyclic, so $F\cap H\cong\mathbb{Z}$.

Since $F\trianglelefteq PB_{n+1}(\mathbb{R}P^2)$, every $g\in N(H)$ satisfies
$g(F\cap H)g^{-1}=gFg^{-1}\cap gHg^{-1}=F\cap H$, whence $N(H)\le N(F\cap H)$ and therefore
\[
F\cap N(H)\ \le\ N_F(F\cap H).
\]
In a free group, the normalizer of a nontrivial cyclic subgroup coincides with its centralizer, which is infinite cyclic; hence $N_F(F\cap H)\cong\mathbb{Z}$ and contains $F\cap H$ as a finite-index subgroup. Consequently $K$ is a subgroup of the finite group $N_F(F\cap H)/(F\cap H)$, and so it is finite.

Since every extension of a finite group is finite, we can apply \cref{exten:gd:inequality} to \cref{s.e.s:254} to obtain
\[
\begin{split}
\gdfin(W(H)) 
&\le 0 + \gdfin(\dfrac{L}{p(H)}) \\
&\le \gdfin(\dfrac{N(p(H))}{p(H)}) \\
&= \gdfin(Wp(H)) \\
&= 0,
\end{split}
\]
where the last inequality is \cref{vcd:centra:finite:0:projec}. This proves the claim in this case, since $0\le \vcd(PB_{n+1}(\mathbb{R}P^2))+r$.

\medskip
\noindent\emph{Case 2: $p(H)$ is infinite.}
As before, restricting $(*)$ to $N(H)$ and passing to the quotient by $H$ we obtain
\[
1\to K \to W(H) \to \dfrac{L}{p(H)} \to 1, \qquad
K:=\dfrac{\pi_1(\mathbb{R}P^2-Q)\cap N(H)}{\pi_1(\mathbb{R}P^2-Q)\cap H}, \quad L=p(N(H))\le N(p(H)).
\]
We distinguish according to $r=\rank(H)$.

\smallskip
\noindent\textit{Subcase $r=1$.}
Since $p(H)$ is infinite and $H$ is virtually $\mathbb{Z}$, the subgroup $\pi_1(\mathbb{R}P^2-Q)\cap H$ is finite; being a subgroup of a free group, it is trivial. 
\[
\begin{split}
\gdfin(W(H))
&\le \gdfin(K)+\gdfin(\dfrac{L}{p(H)}) \\
&\le 0+ \gdfin(\dfrac{N(p(H))}{p(H)}) \\
&= \gdfin(W(p(H))) \\
&\le \vcd(PB_n(\mathbb{R}P^2)) \\
&\leq\vcd(PB_{n+1}(\mathbb{R}P^2)) \\
&\leq \vcd(PB_{n+1}(\mathbb{R}P^2))+r.
\end{split}
\]

\smallskip
\noindent\textit{Subcase $r\ge 2$.}
In this case $K$ is finite, so  $\gdfin{K}=0$. Using \cref{exten:gd:inequality}, the induction hypothesis, and $\rank(p(H))=\rank(H)=r$, together with $\vcd(PB_n(\mathbb{R}P^2))\le\vcd(PB_{n+1}(\mathbb{R}P^2))$, we obtain
\[
\begin{split}
\gdfin(W(H))
&\le \gdfin(K)+\gdfin{\dfrac{N(p(H))}{p(H)}} \\
&\le 0+\gdfin{W{p(H)}} \\
&\le \vcd(PB_n(\mathbb{R}P^2))+\rank(p(H)) \\
&\le \vcd(PB_{n+1}(\mathbb{R}P^2))+\rank(H) \\
&= \vcd(PB_{n+1}(\mathbb{R}P^2))+r.
\end{split}
\]

In both subcases of Case 2 we obtain $\gdfin{W{H}}\le \vcd(PB_{n+1}(\mathbb{R}P^2))+r$, which together with Case 1 completes the inductive step, and hence the proof.
\end{proof}

\begin{theorem}[The upper bound]\label{upper:bound:proyective}
For every $r \in \mathbb{N}$, we have
\[
  \gd_{\mathcal{F}_r}(\pb_n(\mathbb{R}P^2)) \leq \vcd(\pb_n(\mathbb{R}P^2)) + r.
\]
\end{theorem}

\begin{proof}
Denote $G = \pb_n(\mathbb{R}P^2)$. We proceed by induction on $r$. For $r = 0$, the inequality $\gd_{\mathcal{F}_0}(G) \leq \vcd(G)$ holds by \cref{vcd:gdfin:pure:projec}.

Suppose the statement holds for all $r \leq k$, and let us prove the inequality for $r = k+1$.
Define an equivalence relation on $\mathcal{F}_{k+1}- \mathcal{F}_k$ by commensurability,
and let $I$ be a complete set of representatives of conjugacy classes of
$\mathcal{F}_{k+1}- \mathcal{F}_k / {\sim}$ up to conjugation.

By \cref{Luck:weiermann}, the following homotopy $G$-push-out gives a model for $E_{\mathcal{F}_{k+1}}G$:
\[
\begin{tikzcd}
  \displaystyle\coprod_{H \in I} G \times_{N_G[H]} E_{\mathcal{F}_k}N_G[H]
    \arrow[r] \arrow[d] &
  E_{\mathcal{F}_k}G \arrow[d] \\
  \displaystyle\coprod_{H \in I} G \times_{N_G[H]} E_{\mathcal{F}_{k+1}[H]}N_G[H]
    \arrow[r] &
  E_{\mathcal{F}_{k+1}}G
\end{tikzcd}
\]
From this push-out we deduce
\begin{align*}
  \gd_{\mathcal{F}_{k+1}}(G)
    &\leq \max\!\left\{
        \gd_{\mathcal{F}_k}(G) + 1,\;
        \max_{H \in I} \gd_{\mathcal{F}_{k+1}[H]}\!\left(N_G[H]\right)
      \right\} \\
    &\leq \max\!\left\{
        \vcd(G) + k + 1,\;
        \max_{H \in I} \gd_{\mathcal{F}_{k+1}[H]}\!\left(N_G[H]\right)
      \right\},
\end{align*}
where the second inequality follows from the induction hypothesis.
Thus, to complete the inductive step it suffices to show that, for every $H \in I$,
\[
  \gd_{\mathcal{F}_{k+1}[H]}(N_G(H)) \leq \vcd(G) + k + 1.
\]

In \cref{comm:realize:normalizer:pure:projec} we show that for every $H \in I$ there exists $H' \leq G$ such that
$H'$ is commensurable with $H$ and $N[H]=N_G(H')$.
Making this substitution, the inequality above becomes
\[
  \gd_{\mathcal{F}_k[H]}(N_G(H)) \leq \vcd(G) + k + 1.
\]
We may write the family 
\[
  \mathcal{F}_k[H]
    = \bigl\{ K \leq N_G(H) \mid K \in \mathcal{F}_{k+1}- \mathcal{F}_k,\; K \sim H \bigr\}
      \cup \bigl(\mathcal{F}_k \cap N_G(H)\bigr),
\]
 as the union of two families
\[
  \mathcal{F}_k[H] = \mathcal{G} \cup \bigl(\mathcal{F}_{k-1} \cap N_G(H)\bigr),
\]
where $\mathcal{G}$ is generated by
\[
  \bigl\{ K \leq N_G(H) \mid K \in \mathcal{F}_k- \mathcal{F}_{k-1},\;
  [K] = [H] \bigr\}.
\]

By \cref{upper:bound:gd:two:families} we obtain
\begin{align*}
  \gd_{\mathcal{F}_k[H]}(N_G(H))
    &\leq \max\!\left\{
        \gd_{\mathcal{F}_k}(N_G(H)),\;
        \gd_{\mathcal{G} \cap (\mathcal{F}_k \cap N_G(H))}(N_G(H)) + 1,\;
        \gd_{\mathcal{G}}(N_G(H))
      \right\} \\
    &\leq \max\!\left\{
        \vcd(G) + k - 1,\;
        \gd_{\mathcal{G} \cap (\mathcal{F}_k \cap N_G(H))}(N_G(H)) + 1,\;
        \gd_{\mathcal{G}}(N_G(H))
      \right\},
\end{align*}
where the second inequality follows from the induction hypothesis.
We claim the following inequalities hold:
\begin{enumerate}[i)]
  \item $\gd_{\mathcal{G}}(N_G(H)) \leq \vcd(G) + k$,
  \item $\gd_{\mathcal{G} \cap (\mathcal{F}_k \cap N_G(H))}(N_G(H)) \leq \vcd(G) + k - 1$.
\end{enumerate}
Once these are established, we conclude
\[
  \gd_{\mathcal{F}_k[H]}(N_G(H)) \leq \vcd(G) + k + 1.
\]

For (i), we note that a model for $E_{\mathcal{F}_0}W(H)$ is a model for $E_{\mathcal{G}}(N_G(H))$
via the projection $N_G(H) \to W(H)$; the inequality therefore follows from \cref{proper:dim:weyl:pure:projec}.
The proof of (ii) is analogous to that of \cref{upper:bound:gd:fn}.
\end{proof}

\begin{theorem}
For every $r \in \mathbb{N}$, we have
\[
  \gd_{\mathcal{F}_r}(\pb_n(\mathbb{R}P^2)) = \vcd(\pb_n(\mathbb{R}P^2)) + r.
\]
\end{theorem}

\begin{proof}
For the lower bound, by \cref{free:abelian:subgroup:sphere} the group $\pb_n(\mathbb{R}P^2)$ contains a free abelian subgroup of rank $\vcd(\pb_n(\mathbb{R}P^2))$, so monotonicity of the geometric dimension gives
\[
  \gd_{\mathcal{F}_r}(\pb_n(\mathbb{R}P^2)) \geq \vcd(\pb_n(\mathbb{R}P^2)) + r.
\]
The upper bound is \cref{upper:bound:proyective}.
\end{proof}

\section{Classifying spaces of infinite-type surface braid groups}\label{virt:abelian:dimension:surface:braid:groups:inf}
In this section we compute the virtually cyclic dimension of infinite-type surface braid groups. For this we need some lemmas.

\begin{lemma}\label{upper:bound:weil:infinite:type}
Let \( \Sigma \) be a connected surface of infinite type. For every infinite virtually abelian subgroup \( L \) of \( \pb_{m}(\Sigma) \), we have 
\[
\gdfin(W_G(L)) \leq \cd(\pb_{m}(\Sigma)) - \rank(L).
\]
\end{lemma}
\begin{proof}
 The proof is by induction on $m$. When \( m = 1 \), we have \( \pb_1(\Sigma) = \pi_1(\Sigma) \), it is well known that \( \pi_1(\Sigma) \) is infinite countable non-finitely generated free group. Then $L\cong \Z$ and $N_{\pb_1(\Sigma)}(L)\cong \Z$, it follows that   \[
\gdfin(W_G(L)) \leq \cd(\pb_{1}(\Sigma)) - \rank(L).
\]
 Assume the lemma holds for \( m \). We now prove it for \( m + 1 \). By \cref{Fadell:Neuwirth} we have the following short exact sequence
\[
1 \to \pi_1(\Sigma - Q_m) \to \pb_{m+1}(\Sigma) \to \pb_m(\Sigma) \to 1,
\]
where \( Q_m \) denotes the set of \( m \) distinct points on \( \Sigma \). Now the proof follows along the same lines as \cref{upper:bound:weil}.
\end{proof}

\begin{definition}\cite[Condition 4.1]{MR2545612}
We say that a group \( G \) satisfies \emph{condition C} if for every \( g,h \in G \) with \( |h| = \infty \) and \( k,l \in \mathbb{Z} \), the equation \( gh^k g^{-1} = h^l \) implies that \( |k| = |l| \).
\end{definition}

\begin{lemma}\cite[Lemma 4.4]{MR2545612}\label{luck:condition:C}
 Let $n$ be an integer. Suppose  that $G$ satisfies condition C. Suppose that $\gdfin(G)\leq n$ and for every infinite cyclic subgroup $H$ of $G$ we have $\gdfin(W_G(H))\leq n$. Then $\gdvc(G)\leq n+1$.  
\end{lemma}

\begin{proposition}\label{pure:braid:group:condition:c:inifinite}
Let \( \Sigma \) be a connected surface of infinite type. Then \( \pb_{m}(\Sigma) \) and $B_n(\Sigma)$ satisfies condition C.
\end{proposition}
\begin{proof}
By \cite[Lemma 4.14]{MR1760573} it is enough to prove the claim for \( \pb_{m}(\Sigma) \).
The proof proceeds by induction on \( m \). For \( m = 1 \), we have \( \pb_{1}(\Sigma) = \pi_1(\Sigma) \), which is a countable free group and hence satisfies condition~\( C \) see for example \cite[Lemma 2.3]{MR4817678}.

Assume that the claim holds for all \( k \leq m \), and we prove it for \( m+1 \). Consider the short exact sequence:
\[
1 \longrightarrow \pi_1(\Sigma - Q_m) \longrightarrow \pb_{m+1}(\Sigma) \xrightarrow{p} \pb_m(\Sigma) \longrightarrow 1.
\]
Let \( g, h \in \pb_{m+1}(\Sigma) \) with \( |h| = \infty \), and suppose there exist \( k, l \in \mathbb{Z} \) such that
\[
g h^k g^{-1} = h^l.
\]

We distinguish two cases:

\medskip
\noindent
\textbf{Case 1:} \( h \notin \ker(p) \). Then
\[
p(g) \, p(h)^k \, p(g)^{-1} = p(h)^l.
\]
Since \cref{dc:infi:braid} ensures that \( \pb_m(\Sigma) \) is torsion-free, the element \( p(h) \) has infinite order. Moreover, by the inductive hypothesis, \( \pb_m(\Sigma) \) satisfies Condition~\( C \), and hence \( k = \pm l \).

\medskip
\noindent
\textbf{Case 2:} \( h \in \ker(p) = \pi_1(\Sigma - Q_m) \). Conjugation by \( g \) defines a homomorphism
\[
\varphi \colon \pi_1(\Sigma - Q_m) \longrightarrow \pi_1(\Sigma - Q_m),
\]
and \( \varphi(h^k) = h^l \). Note that the cyclic subgroup \( \langle h \rangle \) is contained in a unique maximal infinite cyclic subgroup \( H \leq \pi_1(\Sigma - Q_m) \). Since conjugation preserves maximal cyclic subgroups, we must have \( \varphi(H) = H \). As \( H \cong \mathbb{Z} \), it follows that \( \varphi|_H = \pm \mathrm{id} \), and therefore \( k = \pm l \).

\medskip
\noindent
This concludes the induction step, and hence the proof.
\end{proof}

\begin{theorem}
Let \( \Sigma \) be a connected surface of infinite type. Then, 
\[
\gd_{\calF_1}(\pb_n(\Sigma)) = n+ 1.
\]
\end{theorem}
\begin{proof}
 We first prove the upper bound. By \cref{luck:condition:C} and \cref{pure:braid:group:condition:c:inifinite}, it suffices to verify the following conditions:
\begin{enumerate}
    \item[(i)] $\gd(\pb_n(\Sigma))\leq n$;
    \item[(ii)] for every infinite cyclic subgroup \( H \leq \pb_n(\Sigma) \), we have \( \underline{\gd}(W(H)) \leq n \).
\end{enumerate}

 Item~(i) was proved in \cite[Theorem 4.3]{2025arXiv250610706L}, and item~(ii) follows from \cref{upper:bound:weil:infinite:type}.

\medskip

We now prove the lower bound. Observe that \( \Sigma \) contains a subsurface homeomorphic to \( \Sigma_0^2 \) such that each component of \( \Sigma - \Sigma_0^2 \) is not homeomorphic to a disk. By \cite[Proposition 4.2]{2025arXiv250610706L}, this implies that \( \pb_n(\Sigma_0^2) \) embeds as a subgroup of \( \pb_n(\Sigma) \). The desired lower bound then follows from \cref{free:abelian:subgroup:cd} and \cref{Virtually:dimension:Z}.
   
\end{proof}


\bibliographystyle{alpha} 
\bibliography{mybib}
\end{document}